\documentclass[11pt]{amsart}
\usepackage{amsmath,amssymb,amsthm,mathtools}
\usepackage[margin=1.10in]{geometry}
\usepackage{hyperref}
\usepackage{booktabs}
\hypersetup{colorlinks=true,linkcolor=blue,citecolor=blue,urlcolor=blue}

\theoremstyle{plain}
\newtheorem{theorem}{Theorem}[section]
\newtheorem{lemma}[theorem]{Lemma}
\newtheorem{proposition}[theorem]{Proposition}
\newtheorem{corollary}[theorem]{Corollary}
\theoremstyle{definition}
\newtheorem{definition}[theorem]{Definition}
\newtheorem{remark}[theorem]{Remark}

\DeclareMathOperator{\Cat}{Cat}
\newcommand{\W}{\mathcal{W}}
\newcommand{\N}{\mathbb{N}}
\newcommand{\Q}{\mathbb{Q}}
\newcommand{\Z}{\mathbb{Z}}

\begin{document}

\title{Nonnesting permutations avoiding 123}
\author{Eric Cowan}
\address{Independent researcher}
\email{eric.c.cowan@gmail.com}
\date{September 5, 2026}

\begin{abstract}
A \emph{nonnesting permutation} of the multiset $\{1,1,2,2,\dots,n,n\}$ is a word
containing no subsequence $abba$ with $a \neq b$. Elizalde and Luo enumerated
nonnesting permutations avoiding each set of two or more patterns of length three,
and left open the enumeration of those avoiding the single pattern $123$
(Problem~1 of their paper). We solve this problem. Writing $c_n$ for the number of nonnesting permutations of
$\{1,1,\dots,n,n\}$ avoiding $123$ and $C(z)=\sum_{n\ge 0}c_nz^n$, we prove
\[
  C(z) \;=\; 1 + \frac{z(1+U)}{2 - z(1+U)(1+V)},
  \qquad V = 1+zV^2, \qquad U = 1+3z+zU^2,
\]
where $U$ and $V$ are the formal power series branches with constant term $1$. In
particular $C$ is algebraic of degree exactly $4$ over $\Q(z)$, and
$c_n \sim \frac{9+6\sqrt3}{8\sqrt\pi}\, 6^n n^{-3/2}$. The proof constructs a
generating tree for these objects whose labels consist of two nonnegative
integers, and applies the kernel method twice. By reversal, the same generating
function and asymptotic formula enumerate nonnesting permutations avoiding $321$.
\end{abstract}

\maketitle

\section{Introduction}

Let $\W_n$ denote the set of words $w=w_1w_2\cdots w_{2n}$ using each of the
labels $1,2,\dots,n$ exactly twice, subject to two conditions:
\begin{enumerate}
\item[(N)] $w$ is \emph{nonnesting}: there are no indices $i<j<k<\ell$ and labels
  $a\neq b$ with $w_i=w_\ell=a$ and $w_j=w_k=b$. Equivalently, $w$ avoids both of
  the patterns $1221$ and $2112$.
\item[(A)] $w$ \emph{avoids} $123$: there are no indices $i<j<k$ with
  $w_i<w_j<w_k$.
\end{enumerate}
Set $c_n = |\W_n|$, with $c_0=1$, and
$C(z)=\sum_{n\ge0}c_nz^n$.

Words satisfying (N) alone were introduced by Elizalde and Luo \cite{EL} under the
name \emph{nonnesting permutations} of the multiset $\{1,1,\dots,n,n\}$; there are
$n!\,\Cat(n)=(2n)!/(n+1)!$ of them, where $\Cat(n)$ is the $n$th Catalan number.
Elizalde and Luo enumerated the nonnesting permutations avoiding each set of two or
more patterns of length $3$, as well as some sets of patterns of length $4$. For a
single pattern $\sigma\in S_3$, reversal and complementation both preserve the
nonnesting property, so the problem reduces to the two cases $\sigma=123$ and
$\sigma=132$; Elizalde and Luo left both open, listing the $123$ case as
Problem~1. The $132$ case (equivalently $231$, $213$, $312$) was subsequently
solved by Archer and Laudone \cite{AL}, whose generating function satisfies a cubic
equation over $\Q(z)$. For the remaining pattern, which they state as $321$
(equivalent to $123$ by reversal; see Corollary~\ref{cor:reverse}), they noted
that only the bounds $\Cat(n)\le c_n\le \Cat(n)^2$ were known, and the published
version of \cite{EL} states that no formula for $c_n(123)$ is available. The sequence $(c_n)_{n\ge1}$
is \href{https://oeis.org/A382361}{OEIS A382361}, where eleven terms are listed
and no formula is given.

Our main result is the following.

\begin{theorem}\label{thm:main}
Let $V$ and $U$ be the formal power series with constant term $1$ satisfying
$V=1+zV^2$ and $U=1+3z+zU^2$. Then
\[
  C(z) \;=\; 1 + \frac{z(1+U)}{2 - z(1+U)(1+V)}.
\]
\end{theorem}

Explicitly, $V=\sum_{n\ge0}\Cat(n)z^n$ is the Catalan generating function, while
\[
  U \;=\; \frac{1-\sqrt{(1-6z)(1+2z)}}{2z}
    \;=\; 1+4z+8z^2+32z^3+128z^4+576z^5+\cdots .
\]
Expanding,
\[
C(z)=1+z+4z^2+17z^3+82z^4+406z^5+2070z^6+10729z^7+56394z^8+\cdots,
\]
extending A382361 by, among others,
$c_{12}=47245644$, $c_{13}=258581288$, $c_{14}=1422695978$.

\begin{corollary}\label{cor:alg}
$C$ is algebraic of degree exactly $4$ over $\Q(z)$: it is the unique formal power
series with $C(0)=1$ and $[z]C=1$ satisfying
\begin{multline*}
  (1-4z+4z^2) + (-4+16z-16z^2)\,C + (5-22z+36z^2+8z^3)\,C^2 \\
  {}+ (-2+12z-40z^2-16z^3)\,C^3 + (-2z+17z^2+12z^3+4z^4)\,C^4 \;=\; 0.
\end{multline*}
\end{corollary}

\begin{corollary}\label{cor:asy}
As $n\to\infty$,
\[
  c_n \;\sim\; \frac{9+6\sqrt3}{8\sqrt{\pi}}\; 6^n\, n^{-3/2}
  \;=\; 1.36761\ldots\times 6^n n^{-3/2}.
\]
\end{corollary}

\begin{corollary}\label{cor:reverse}
Let $d_n$ be the number of nonnesting permutations of
$\{1,1,\dots,n,n\}$ avoiding $321$. Then $d_n=c_n$ for every $n\geq0$.
Consequently, their generating function is $C(z)$, it has algebraic degree
exactly $4$ over $\Q(z)$, and
\[
  d_n \;\sim\; \frac{9+6\sqrt3}{8\sqrt{\pi}}\;6^n n^{-3/2}.
\]
\end{corollary}

\begin{proof}
Reversal is an involution. It preserves nonnesting because the forbidden
configuration $abba$ reads the same backwards, and it interchanges $123$ and
$321$ subsequences. Thus it restricts to the required size-preserving
bijection; the remaining claims follow from Theorem~\ref{thm:main} and
Corollaries~\ref{cor:alg} and~\ref{cor:asy}.
\end{proof}

Corollary~\ref{cor:asy} sharpens the bounds $\Cat(n)\le c_n\le \Cat(n)^2$ of
\cite{AL}, which confine the exponential growth rate to the interval $[4,16]$, to
the exact value $6$.

The proof of Theorem~\ref{thm:main} occupies Sections~\ref{sec:tree}--\ref{sec:kernel}.
In Section~\ref{sec:tree} we construct a generating tree for $\bigcup_n\W_n$ in
which each object carries a label $(r,s)\in\N^2$ recording the shape of its
longest weakly decreasing prefix, and we prove that the children of an object
depend only on its label. Section~\ref{sec:eqs} converts the resulting succession
rule into functional equations, and Section~\ref{sec:kernel} solves them by two
applications of the kernel method. Section~\ref{sec:verif} records the
computational checks.

\section{Preliminaries}\label{sec:prelim}

Throughout, $w\in\W_n$ is read left to right. For a label $a$ we write $p_1(a)<p_2(a)$
for the two positions carrying $a$, and call the pair $\{p_1(a),p_2(a)\}$ the
\emph{arc} of $a$. An arc is \emph{open} at position $k$ if $p_1(a)<k\le p_2(a)$;
informally, reading position by position, a first occurrence opens an arc and a
second occurrence closes one.

\begin{lemma}[FIFO closure]\label{lem:fifo}
Let $w$ be nonnesting and let $k$ be a position carrying a second occurrence.
Among the arcs open immediately before $k$, the one closed at position $k$ is the
one opened earliest.
\end{lemma}

\begin{proof}
Suppose arcs $a$ and $b$ are both open immediately before $k$, with
$p_1(a)<p_1(b)$, and that $p_2(b)=k$. Then $p_2(a)>k=p_2(b)$, so
$p_1(a)<p_1(b)<p_2(b)<p_2(a)$, which is a nesting.
\end{proof}

Consequently a nonnesting word is determined by its underlying Dyck path (the
sequence of openings and closings) together with the assignment of labels to the
openings, which recovers $|\{w : \text{(N)}\}| = n!\,\Cat(n)$.

\begin{definition}
The \emph{event word} of a prefix $w_1\cdots w_L$ is the word $e_1\cdots e_L$ over
$\{U,D\}$ with $e_i=U$ if $w_i$ is a first occurrence in $w$ and $e_i=D$ otherwise.
\end{definition}

We will use the following elementary consequence of (A) repeatedly: if $w$ avoids
$123$ and $w_i<w_j$ for some $i<j$, then no letter after position $j$ exceeds
$w_j$; in particular every letter following an ascent is at most the top of that
ascent.

\section{A generating tree}\label{sec:tree}

\subsection{The label}

\begin{definition}\label{def:label}
For $w\in\W_n$ let $L=L(w)$ be the length of the longest weakly decreasing prefix
of $w$, that is, the largest $L$ with $w_1\ge w_2\ge\cdots\ge w_L$. Let $e$ be the
event word of $w_1\cdots w_L$.
\end{definition}

\begin{lemma}\label{lem:shape}
For every $w\in\W_n$ the event word $e$ of Definition~\ref{def:label} has the form
$(UD)^rU^s$ for unique integers $r,s\ge0$, and $2r+s=L$.
\end{lemma}

\begin{proof}
We first show that no $D$ occurs in $e$ at a moment when two or more arcs are
open. Suppose $e_k=D$ and that immediately before position $k$ the open arcs were
opened at positions $i_1<i_2<\cdots<i_m$ with $m\ge2$. The labels
$w_{i_1},\dots,w_{i_m}$ are pairwise distinct, and $i_1<\cdots<i_m\le L$, so by
weak decrease $w_{i_1}\ge\cdots\ge w_{i_m}$ and hence $w_{i_1}>w_{i_m}$. By
Lemma~\ref{lem:fifo}, $w_k=w_{i_1}$. But $i_m<k\le L$, so weak decrease gives
$w_k\le w_{i_m}<w_{i_1}=w_k$, a contradiction.

Let $h_k$ denote the number of arcs open after the first $k$ events, so $h_0=0$,
a $U$ raises $h$ by one and a $D$ lowers it by one. The preceding paragraph shows
that if $e_k=D$ then $h_{k-1}=1$, and consequently $h_k=0$. Now read $e$ from the
left. From height $0$ the next event, if any, is a $U$ (a $D$ would need
$h_{k-1}=1$), bringing the height to $1$. If the following event is a $D$, the two
events form a block $UD$ and the height returns to $0$. If instead it is a $U$,
the height becomes $2$; since a $D$ can only occur from height $1$ and the height
never again drops, every remaining event is a $U$. The event word may also end
while the height is $0$ or $1$. It follows that
$e=(UD)^rU^s$ for unique integers $r,s\ge0$ (with $s=0$, $s=1$, or $s\ge2$
according as the word ends at height $0$, ends at height $1$, or reaches height
$2$), and necessarily $2r+s=L$.
\end{proof}

\begin{definition}
The \emph{label} of $w\in\W_n$ is the pair $\ell(w)=(r,s)$ of
Lemma~\ref{lem:shape}.
\end{definition}

Note $r\le n$ and $s\le n$, since $r$ counts disjoint arcs and $s$ counts distinct
labels. The empty word (the unique element of $\W_0$) has label $(0,0)$, and this
is the only object with that label: for $n\ge1$ we have $L\ge1$, so $(r,s)\ne(0,0)$.
The unique element $11$ of $\W_1$ has label $(1,0)$.

\subsection{Insertions}

Every $w\in\W_{n+1}$ determines a \emph{parent} $\pi(w)\in\W_n$, obtained by
deleting both copies of the maximal label $n+1$: conditions (N) and (A) are
inherited by subwords, and the result uses each of $1,\dots,n$ twice. Conversely
the children of $w\in\W_n$ are the elements of $\W_{n+1}$ obtained by inserting two
copies of $n+1$ into $w$. It is convenient to index insertions by gaps: for
$0\le i\le j\le 2n$, let $w^{(i,j)}$ denote the word obtained from $w$ by
inserting a copy of $n+1$ immediately before position $i+1$ and a second copy
immediately before position $j+1$; the two new letters occupy final positions $i$
and $j+1$ (zero-based). Every child arises from exactly one pair $(i,j)$.

\begin{lemma}\label{lem:insert}
Let $w\in\W_n$ have label $(r,s)$ and set $L=2r+s$. Then $w^{(i,j)}\in\W_{n+1}$ if
and only if all of the following hold:
\begin{enumerate}
\item[(i)] $j\le L$;
\item[(ii)] $i\le 2r$ whenever $s\ge1$;
\item[(iii)] there is no $0\le k\le r-1$ with $i\le 2k$ and $j\ge 2k+2$;
\item[(iv)] $(i,j)\ne(2k+1,2k+1)$ for every $0\le k\le r-1$.
\end{enumerate}
\end{lemma}

\begin{proof}
\emph{Condition (A).} Since $n+1$ is a strict maximum and a $123$ pattern requires
two strict inequalities, no copy of $n+1$ can play the role of the ``$1$'' or the
``$2$'', and two copies of $n+1$ cannot both occur in one pattern. Hence every
$123$ pattern of $w^{(i,j)}$ not already present in $w$ consists of an ascent of
$w$ followed by a copy of $n+1$. As $w$ itself avoids $123$, the word $w^{(i,j)}$
avoids $123$ if and only if no ascent of $w$ lies entirely before the \emph{second}
copy of $n+1$. The letters of $w$ preceding that copy are $w_1\cdots w_j$, so the
condition is that $w_1\cdots w_j$ be weakly decreasing, i.e.\ $j\le L$. This is (i).

\emph{Condition (N).} Since $w$ is nonnesting, we need only compare the new arc
with the arcs of $w$. Let $a$ be a label of $w$ with arc $\{p_1,p_2\}$ (positions
in $w$, one-based, so gaps $p-1$ and $p$ surround position $p$). The new arc
contains the arc of $a$ if and only if $i\le p_1-1$ and $j\ge p_2$; the arc of $a$
contains the new arc if and only if $i\ge p_1$ and $j\le p_2-1$.

By (i) we have $j\le L$, so both insertions lie inside the weakly decreasing
prefix, whose event word is $(UD)^rU^s$. An arc of $w$ with $p_1>L$ can interact
with the new arc in neither way, since containing the new arc would need
$i\ge p_1>L\ge j\ge i$, and being contained would need $j\ge p_2>L$. The
remaining arcs, those meeting $w_1\cdots w_L$, are of two kinds. First, the $r$ \emph{completed} arcs, occupying positions $2k+1,2k+2$
for $0\le k\le r-1$; in gap coordinates these give: the new arc contains the $k$th
completed arc iff $i\le 2k$ and $j\ge 2k+2$, which is (iii); and the $k$th
completed arc contains the new arc iff $i\ge 2k+1$ and $j\le 2k+1$, which combined
with $i\le j$ forces $i=j=2k+1$, which is (iv). Second, the $s$ arcs opened at
positions $2r+1,\dots,2r+s$ and closed after position $L$. Such an arc contains
the new arc iff $i\ge 2r+t$ and $j\le p_2-1$ for the relevant $t\ge1$; since
$p_2>L\ge j$ the second inequality is automatic, so the binding case is $t=1$,
giving the prohibition $i\ge 2r+1$, i.e.\ (ii). Conversely the new arc cannot
contain such an arc, as that would require $j\ge p_2>L$.
\end{proof}

\begin{lemma}\label{lem:childlabel}
Let $w\in\W_n$ have label $(r,s)$ and let $w^{(i,j)}\in\W_{n+1}$.
\begin{enumerate}
\item If $i\ge1$ then $\ell(w^{(i,j)})=(k,0)$ if $i=2k$, and $\ell(w^{(i,j)})=(k,1)$
  if $i=2k+1$.
\item If $i=j=0$ then $\ell(w^{(i,j)})=(r+1,s)$.
\item If $i=0$ and $j\ge1$ then $\ell(w^{(i,j)})=(0,j+1)$.
\end{enumerate}
\end{lemma}

\begin{proof}
Write $m=n+1$ for the new maximum and $w'=w^{(i,j)}$.

(1) The first $i$ letters of $w'$ are $w_1\cdots w_i$, which is weakly decreasing
because $i\le 2r\le L$ by Lemma~\ref{lem:insert}(ii) and (iii) (if $s=0$ then
$i\le j\le L=2r$). The next letter of $w'$ is $m>w_i$, so the longest weakly
decreasing prefix of $w'$ is exactly $w_1\cdots w_i$. Its event word is the length
$i$ prefix of $(UD)^rU^s$, since a letter of $w_1\cdots w_i$ is a first occurrence
in $w'$ precisely when it is one in $w$. Truncating $(UD)^rU^s$ to $2k$ letters
gives $(UD)^k$, and to $2k+1$ letters gives $(UD)^kU$; note $2k+1\le 2r$ forces
$k\le r-1$, so the truncation does indeed lie in the $(UD)^r$ portion.

(2) Here $w'=mm\,w_1w_2\cdots$. Since $m$ is the maximum, $m\ge m\ge w_1$ and the
weakly decreasing prefix of $w'$ is $mm$ followed by the weakly decreasing prefix
of $w$, of length $L+2$. Its event word is $UD$ followed by $(UD)^rU^s$, namely
$(UD)^{r+1}U^s$.

(3) Here $w'=m\,w_1\cdots w_j\,m\cdots$. As $j\le L$ the letters $w_1\cdots w_j$
are weakly decreasing and bounded above by $m$; the letter following them is $m$,
and $m>w_j$, so the longest weakly decreasing prefix of $w'$ is $m\,w_1\cdots w_j$,
of length $j+1$. By Lemma~\ref{lem:insert}(iii) with $k=0$, if $r\ge1$ then
$j\le1$; and if $r=0$ then the event word of $w_1\cdots w_j$ is $U^j$. In either
case the event word of the prefix of $w'$ is $U$ followed by the length $j$ prefix
of $(UD)^rU^s$, which for $j\le1$ or $r=0$ equals $U^{j+1}$. Hence
$\ell(w')=(0,j+1)$.
\end{proof}

\subsection{The succession rule}

\begin{theorem}\label{thm:succ}
The map $\pi$ makes $\bigcup_{n\ge0}\W_n$ into a generating tree rooted at the
empty word, in which the multiset of labels of the children of an object depends
only on its own label, according to the following rule. Exponents denote
multiplicities.
\begin{align*}
(0,0) &\longrightarrow (1,0),\\
(0,s) &\longrightarrow (1,s),\,(0,2),\,(0,3),\dots,(0,s+1) && (s\ge1),\\
(r,s) &\longrightarrow (r+1,s),\,(0,2),\,
   \bigl[(k,1)^2,(k+1,0)^2\bigr]_{k=0}^{r-2},\,
   (r-1,1)^{s+1},\,(r,0)^{s+1} && (r\ge1,\ s\ge0).
\end{align*}
In the last line the bracketed range is empty when $r=1$.
\end{theorem}

\begin{proof}
That $\pi$ is well defined and that each $w\in\W_{n+1}$ has the unique parent
$\pi(w)$ was observed above, so it remains to enumerate the pairs $(i,j)$ allowed
by Lemma~\ref{lem:insert} and apply Lemma~\ref{lem:childlabel}.

Suppose first $r=0$. If also $s=0$ then $L=0$ and $(i,j)=(0,0)$ is the only
option, giving the child $(1,0)$ by Lemma~\ref{lem:childlabel}(2). If $s\ge1$ then
(ii) forces $i=0$, conditions (iii) and (iv) are vacuous, and (i) allows
$0\le j\le s$. The choice $j=0$ gives $(1,s)$ and the choices $j=1,\dots,s$ give
$(0,2),\dots,(0,s+1)$.

Now suppose $r\ge1$. If $s\ge1$ then (ii) gives $i\le 2r$; if $s=0$ then
$i\le j\le L=2r$; so in all cases $0\le i\le 2r$. We enumerate by $i$.

\emph{Case $i=0$.} Condition (iii) with $k=0$ forbids $j\ge2$, so $j\in\{0,1\}$,
giving the children $(r+1,s)$ and $(0,2)$.

\emph{Case $i=2k$ with $1\le k\le r-1$.} The smallest $k'$ with $i\le 2k'$ is
$k'=k$, so (iii) gives $j\le 2k+1$; conditions (i) and (iv) impose nothing further,
and $j\ge i$. Hence $j\in\{2k,2k+1\}$, two children, each of label $(k,0)$.

\emph{Case $i=2r$.} There is no completed arc with index $\ge r$, so (iii) is
vacuous, and (i) gives $2r\le j\le 2r+s$: exactly $s+1$ children, each of label
$(r,0)$.

\emph{Case $i=2k+1$ with $0\le k\le r-2$.} Condition (iv) excludes $j=2k+1$; the
smallest $k'$ with $i\le 2k'$ is $k'=k+1$, so (iii) gives $j\le 2k+3$. Hence
$j\in\{2k+2,2k+3\}$, two children, each of label $(k,1)$.

\emph{Case $i=2r-1$.} Condition (iv) excludes $j=2r-1$, there is no completed arc
with index $\ge r$, and (i) gives $j\le 2r+s$. Hence $2r\le j\le 2r+s$: exactly
$s+1$ children, each of label $(r-1,1)$.

Collecting: $(r+1,s)$ and $(0,2)$ once each; $(k,0)$ twice for $1\le k\le r-1$;
$(k,1)$ twice for $0\le k\le r-2$; $(r,0)$ and $(r-1,1)$ with multiplicity $s+1$
each. Reindexing $(k,0)$ as $(k+1,0)$ for $0\le k\le r-2$ gives the stated rule.
\end{proof}

\section{Functional equations}\label{sec:eqs}

For $n\ge1$ let $f_{n,r,s}$ be the number of $w\in\W_n$ with $\ell(w)=(r,s)$, and
put
\[
  A(v)=\sum_{n\ge1}\sum_{s\ge1} f_{n,0,s}\,z^nv^s,\qquad
  B(u,v)=\sum_{n\ge1}\sum_{r\ge1}\sum_{s\ge0} f_{n,r,s}\,z^nu^rv^s,
\]
\[
  P=A(1),\qquad Q=B(1,1),\qquad G(u)=B(u,0),\qquad \alpha=[v]A(v),
\]
so that $C=1+P+Q$. Since $r\le n$ and $s\le n$, the coefficient of $z^n$ in $A$ and
in $B$ is a polynomial in $v$, resp.\ in $(u,v)$, of degree at most $n$ in each
variable; consequently $A$ and $B$ may be evaluated at any formal power series
argument, and all substitutions below take place in $\Q[[z]]$.

Before deriving the equations we illustrate the translation mechanism on a single
parent. An object of size $n$ with label $(r,s)$, $r\ge1$, contributes the
monomial $z^nu^rv^s$ to $B(u,v)$. By Theorem~\ref{thm:succ} it produces at size
$n+1$, among others, one child of label $(r+1,s)$ and one child of label
$(0,2)$; these children contribute $z^{n+1}u^{r+1}v^{s}=(zu)\cdot z^nu^rv^s$
and $z^{n+1}v^2$ respectively. Summing over all parents, the first transition
family contributes the term $zu\,B(u,v)$ to the equation for $B$, and the second
contributes $zv^2\,B(1,1)=zv^2Q$ to the equation for $A$. Every term in
equations \eqref{eq:1}, \eqref{eq:4} and \eqref{eq:5} below arises in exactly
this way; transitions with multiplicity $s+1$ produce, after summation, the
terms involving $\partial_vB(u,1)$.

\begin{proposition}\label{prop:A}
\begin{equation}\label{eq:1}
  A(v) \;=\; \alpha v \;+\; \frac{zv^2}{1-v}\bigl(P-A(v)\bigr) \;+\; zv^2 Q .
\end{equation}
\end{proposition}

\begin{proof}
By Theorem~\ref{thm:succ} an object of label $(0,s')$ with $s'\ge1$ contributes one
child of each label $(0,t)$ for $2\le t\le s'+1$, an object of label $(r,s)$ with
$r\ge1$ contributes one child of label $(0,2)$, and no other object contributes a
child whose first label coordinate is $0$ and whose second exceeds $1$. Now
\[
\frac{zv^2}{1-v}\bigl(P-A(v)\bigr)
=\sum_{n\ge1}\sum_{s'\ge1}f_{n,0,s'}z^{n+1}\frac{v^2(1-v^{s'})}{1-v}
=\sum_{n\ge1}\sum_{s'\ge1}f_{n,0,s'}z^{n+1}\bigl(v^2+\cdots+v^{s'+1}\bigr),
\]
which is exactly the first family, and $zv^2Q$ is exactly the second. The
remaining part of $A(v)$ is its coefficient of $v^1$, which is $\alpha v$ by
definition.
\end{proof}

\begin{proposition}\label{prop:ids}
For all admissible indices,
\begin{align}
  f_{n,1,0} &= [n=1] + f_{n,0,1}, \label{eq:id1}\\
  f_{n,r,1} &= f_{n,r+1,0} - [n=r+1] \qquad (r\ge1), \label{eq:id2}\\
  f_{n,r,s} &= f_{n-r,0,s} \qquad (r\ge1,\ s\ge2). \label{eq:id3}
\end{align}
\end{proposition}

\begin{proof}
\eqref{eq:id3}: by Theorem~\ref{thm:succ}, a child of label $(r,s)$ with $r\ge1$
and $s\ge2$ arises only from the transition $(r-1,s)\to(r,s)$, since every other
child listed has second coordinate $0$, $1$, or $2$ with first coordinate $0$.
Hence $f_{n,r,s}=f_{n-1,r-1,s}$, and iterating $r$ times gives \eqref{eq:id3}.

\eqref{eq:id1}: children of label $(1,0)$ arise from $(0,0)\to(1,0)$ (only at
$n=1$), from $(k+1,0)^2$ with $k=0$, which requires a parent $(R,S)$ with $R\ge2$,
and from $(R,0)^{S+1}$ with $R=1$. Children of label $(0,1)$ arise from $(k,1)^2$
with $k=0$, again requiring $R\ge2$, and from $(R-1,1)^{S+1}$ with $R=1$. The
transition $(R+1,S)$ contributes to $(1,0)$ only if $R=0$, i.e.\ from the root,
already counted. The two counts therefore agree apart from the root contribution.

\eqref{eq:id2}: fix $r\ge1$ and compare, for a single parent of label $(R,S)$, the
number of children labelled $(r+1,0)$ with the number labelled $(r,1)$. If $R\ge1$,
Theorem~\ref{thm:succ} gives
\begin{align*}
  \#(r+1,0) &= 2[r\le R-2]+(S+1)[r=R-1]+[S=0][R=r],\\
  \#(r,1)   &= 2[r\le R-2]+(S+1)[r=R-1]+[S=1][R=r-1],
\end{align*}
so that
\[
  \#(r+1,0)-\#(r,1)=[S=0][R=r]-[S=1][R=r-1];
\]
here the terms $[S=0][R=r]$ and
$[S=1][R=r-1]$ come from the transition $(R,S)\to(R+1,S)$. If $R=0$ and $S\ge1$
the only child is $(1,S)$ together with children of the form $(0,t)$, so
$\#(r+1,0)-\#(r,1)=-[S=1][r=1]$; and the root contributes $0$ for $r\ge1$.
Summing over the objects of size $n-1$,
\[
  f_{n,r+1,0}-f_{n,r,1} = f_{n-1,r,0}-f_{n-1,r-1,1},
\]
valid for all $r\ge1$ (for $r=1$ the term $f_{n-1,0,1}$ arises from the $R=0$
parents). We induct on $n$. For $n=1$ both sides of \eqref{eq:id2} vanish for every
$r\ge1$, since $f_{1,r,s}=[r=1][s=0]$. For $n\ge2$ and $r\ge2$, the inductive
hypothesis at $(n-1,r-1)$ gives $f_{n-1,r-1,1}=f_{n-1,r,0}-[n-1=r]$, whence
$f_{n,r+1,0}-f_{n,r,1}=[n=r+1]$ as required. For $r=1$ we use \eqref{eq:id1}
instead: $f_{n-1,1,0}-f_{n-1,0,1}=[n-1=1]=[n=2]$, giving the claim.
\end{proof}

\begin{proposition}\label{prop:B}
\begin{equation}\label{eq:4}
  B(u,v) \;=\; G(u) \;+\; v\Bigl(\frac{G(u)}{u}-(z+\alpha)\Bigr)
  \;+\; \frac{zu}{1-zu}\bigl(A(v)-(z+\alpha)v\bigr).
\end{equation}
\end{proposition}

\begin{proof}
We compare coefficients of powers of $v$. The coefficient of $v^0$ on the right is
$G(u)$, since $A(v)$ has no constant term in $v$; this matches $B(u,0)=G(u)$.

For $v^1$: since $G(u)=\sum_{n\ge1}\sum_{r\ge1}f_{n,r,0}z^nu^r$ we have
$G(u)/u=\sum_{n,r\ge1}f_{n,r,0}z^nu^{r-1}$, whose $u^0$ term is
$\sum_n f_{n,1,0}z^n=z+\alpha$ by \eqref{eq:id1}. Hence
$G(u)/u-(z+\alpha)=\sum_{n\ge1}\sum_{r\ge1}f_{n,r+1,0}z^nu^r$. The coefficient of
$v^1$ in $A(v)-(z+\alpha)v$ is $\alpha-(z+\alpha)=-z$, so the last term of
\eqref{eq:4} contributes $-z^2u/(1-zu)=-\sum_{r\ge1}z^{r+1}u^r$ to the coefficient
of $v^1$. Adding,
\[
  [v^1]\text{RHS}=\sum_{n\ge1}\sum_{r\ge1}\bigl(f_{n,r+1,0}-[n=r+1]\bigr)z^nu^r
  =\sum_{n\ge1}\sum_{r\ge1}f_{n,r,1}z^nu^r=[v^1]B(u,v)
\]
by \eqref{eq:id2}.

For $v^s$ with $s\ge2$: only the last term contributes, giving
$\bigl(\sum_{r\ge1}z^ru^r\bigr)\bigl(\sum_{m\ge1}f_{m,0,s}z^m\bigr)
=\sum_{r\ge1}\sum_{n>r}f_{n-r,0,s}z^nu^r$, which equals
$\sum_{r\ge1}\sum_n f_{n,r,s}z^nu^r$ by \eqref{eq:id3}.
\end{proof}

\begin{proposition}\label{prop:G}
\begin{equation}\label{eq:5}
  G(u) \;=\; zu \;+\; zu\,G(u) \;+\; \frac{2z}{1-u}\bigl(uQ-B(u,1)\bigr)
  \;+\; z\bigl(B(u,1)+\partial_vB(u,1)\bigr).
\end{equation}
\end{proposition}

\begin{proof}
We collect the children whose label has second coordinate $0$ and first coordinate
$\ge1$. From Theorem~\ref{thm:succ}: the root contributes the single child
$(1,0)$, giving $zu$. A parent of label $(R,0)$ contributes the child $(R+1,0)$,
giving $zuG(u)$. A parent of label $(R,S)$ with $R\ge1$ contributes $(k+1,0)$ with
multiplicity $2$ for $0\le k\le R-2$, that is, $(j,0)$ twice for $1\le j\le R-1$;
since $\sum_{j=1}^{R-1}u^j=(u-u^R)/(1-u)$, these contribute
$\frac{2z}{1-u}\sum_{n,R\ge1,S\ge0}f_{n,R,S}z^n(u-u^R)
=\frac{2z}{1-u}\bigl(uQ-B(u,1)\bigr)$. Finally such a parent contributes $(R,0)$
with multiplicity $S+1$, giving
$z\sum f_{n,R,S}z^n(S+1)u^R=z\bigl(B(u,1)+\partial_vB(u,1)\bigr)$. A parent of
label $(0,S)$ with $S\ge1$ contributes no child with second coordinate $0$.
\end{proof}

\section{The kernel method}\label{sec:kernel}

\subsection{First kernel}

Multiplying \eqref{eq:1} by $1-v$ gives
\begin{equation}\label{eq:1k}
  A(v)\bigl(1-v+zv^2\bigr) \;=\; \alpha v(1-v) + zv^2P + zv^2Q(1-v).
\end{equation}
Let $V\in\Q[[z]]$ be the branch of $V=1+zV^2$ with $V(0)=1$, i.e.\ the Catalan
generating function. Then $1-V+zV^2=1-V+(V-1)=0$, so substituting $v=V$ into
\eqref{eq:1k} — legitimate because $[z^n]A(v)$ is a polynomial in $v$ — yields
\[
  0=\alpha V(1-V)+zV^2P+zV^2Q(1-V) = -(V-1)\bigl(\alpha V - P + Q(V-1)\bigr),
\]
using $zV^2=V-1$. Since $V-1=z+O(z^2)\neq0$ and $\Q[[z]]$ is an integral domain,
\begin{equation}\label{eq:2}
  P=\alpha V+(V-1)Q .
\end{equation}

\begin{proposition}\label{prop:consequences}
Let $T=C-1=P+Q$. Then
\begin{equation}\label{eq:3}
  Q=\frac{T}{V}-\alpha,\qquad P-\alpha = zVT,\qquad A'(1)=\frac{P-\alpha}{z}-Q .
\end{equation}
\end{proposition}

\begin{proof}
From \eqref{eq:2}, $T=P+Q=\alpha V+VQ$, giving the first identity; then
$P-\alpha=T-Q-\alpha=T(1-1/V)=T(V-1)/V=zVT$ since $(V-1)/V=zV$. For the third,
differentiate \eqref{eq:1k} with respect to $v$ and set $v=1$:
\[
  A'(1)\cdot z + A(1)(-1+2z) = -\alpha + 2zP - zQ,
\]
and $A(1)=P$, so $zA'(1)=P-\alpha-zQ$.
\end{proof}

\subsection{Second kernel}

Substituting \eqref{eq:4} into \eqref{eq:5} (using
$B(u,1)=G(u)(1+u^{-1})-(z+\alpha)+\frac{zu}{1-zu}(P-z-\alpha)$ and
$\partial_vB(u,1)=G(u)/u-(z+\alpha)+\frac{zu}{1-zu}(A'(1)-z-\alpha)$) and clearing
the denominators $u$, $1-u$ and $1-zu$ produces an identity of the form
\begin{equation}\label{eq:5k}
  \bigl(1-zu\bigr)\bigl(1+3z-u+zu^2\bigr)\,G(u) \;=\; zu\,\Phi(u),
\end{equation}
where
\[
  \Phi(u)= zA'(1)(1-u)-zP(1+u)+2Q(1-zu)+2\alpha + zu^2 - zu - u + 2z + 1 .
\]
In particular the kernel is $(1+3z-u+zu^2)/(1-u)$ up to the invertible factor
$1-zu$.

Let $U\in\Q[[z]]$ be the branch of $U=1+3z+zU^2$ with $U(0)=1$, explicitly
$U=\bigl(1-\sqrt{(1-6z)(1+2z)}\bigr)/(2z)$. Then $1+3z-U+zU^2=0$. Substituting
$u=U$ into \eqref{eq:5k} is legitimate, since after clearing denominators every
coefficient of $z^n$ on both sides is a polynomial in $u$. The left side vanishes,
and $zU\neq0$, so $\Phi(U)=0$. Using $zU^2=U-1-3z$ to reduce,
\[
  zA'(1)(1-U) - zP(1+U) + 2Q(1-zU) + 2\alpha - z(1+U) = 0 ,
\]
since $zU^2-zU-U+2z+1=(U-1-3z)-zU-U+2z+1=-z(1+U)$.
Substituting the three identities \eqref{eq:3}, the unknown $\alpha$ cancels
identically. Reducing the result with $zV^2=V-1$ and multiplying through by $V$
leaves
\begin{equation}\label{eq:final}
  T\bigl(1+U+V-UV\bigr) - zTV(1+U) - zV(1+U) = 0,
\end{equation}
that is $T=\dfrac{zV(1+U)}{1+U+V-UV-zV(1+U)}$. Multiplying numerator and
denominator by $1/V$ and using $zV^2=V-1$ in the form $(V-1)/V=zV$, this
equals $z(1+U)/\bigl(2-z(1+U)(1+V)\bigr)$, and $C=1+T$ gives
Theorem~\ref{thm:main}. Note $2-z(1+U)(1+V)$ has constant term $2$, so the
quotient is a well-defined element of $\Q[[z]]$.

\begin{proof}[Proof of Corollary~\ref{cor:alg}]
Put $L(V)=T\bigl(2-z(1+U)(1+V)\bigr)-z(1+U)$ with $T=C-1$. Eliminating $V$ between
$zV^2-V+1$ and $L(V)$, and then $U$ between the result and $zU^2-U+1+3z$, the
iterated resultant equals $4z^4$ times the quartic of Corollary~\ref{cor:alg};
since $z\neq0$ in $\Q(z)$, that quartic annihilates $C$.

We claim it is the minimal polynomial of $C$, so that the degree is exactly $4$.
Set
\[
  \sigma_1 = \sqrt{(1-6z)(1+2z)},\qquad \sigma_2=\sqrt{1-4z},
\]
so that $U=(1-\sigma_1)/(2z)$ and $V=(1-\sigma_2)/(2z)$. The polynomials
$(1-6z)(1+2z)$, $(1-4z)$ and their product $(1-6z)(1+2z)(1-4z)$ all have simple
roots (namely $1/6,-1/2$; $1/4$; and $1/6,-1/2,1/4$), hence none is a square in
$\Q(z)$. Therefore $K=\Q(z)(\sigma_1,\sigma_2)$ is a biquadratic extension of
$\Q(z)$ with Galois group $(\Z/2\Z)^2$, generated by the involutions
$\tau_1:\sigma_1\mapsto-\sigma_1$ (fixing $\sigma_2$) and
$\tau_2:\sigma_2\mapsto-\sigma_2$ (fixing $\sigma_1$).

Now $C=1+T$ with $T=z(1+U)/\bigl(2-z(1+U)(1+V)\bigr)\in K$. Expanding the four
conjugates of $T$ as Laurent series at $z=0$:
\begin{align*}
  T &= z+4z^2+17z^3+82z^4+\cdots,\\
  \tau_1(T) &= -z^{-1}-6-47z-430z^2-3931z^3-\cdots,\\
  \tau_2(T) &= -\tfrac12-\tfrac14 z+\tfrac{11}{8}z^2+\tfrac{67}{16}z^3+\cdots,\\
  \tau_1\tau_2(T) &= -z-2z^2-5z^3-22z^4-\cdots.
\end{align*}
These four series are pairwise distinct, so no nontrivial element of
$\operatorname{Gal}(K/\Q(z))$ fixes $T$. Hence $\Q(z)(C)=\Q(z)(T)=K$ has degree
$4$ over $\Q(z)$, the annihilating quartic above is (a scalar multiple of) the
minimal polynomial of $C$, and in particular it is irreducible over $\Q(z)$.

Finally, the leading coefficient $-2z+17z^2+12z^3+4z^4$ of the quartic vanishes at
$z=0$, and the specialisation of the quartic at $z=0$ is $-(C-1)^2(2C-1)$; thus
$C=1$ is a double root there and the condition $C(0)=1$ alone does not determine
the branch. There are exactly two formal power series solutions with constant term
$1$, namely $C$ and $1+\tau_1\tau_2(T)=1-z-2z^2-5z^3-22z^4-90z^5-\cdots$,
distinguished by $[z]C=1$.
\end{proof}

\begin{proof}[Proof of Corollary~\ref{cor:asy}]
Write $D(z)=2-z(1+U(z))(1+V(z))$, so that $C=1+z(1+U)/D$.

The series $V$ has radius of convergence $1/4$ and $U$ has radius of convergence
$1/6$, the latter with a square-root branch point at $z=1/6$; both have
nonnegative coefficients, and both converge at $z=1/6$, where $U(1/6)=3$ and
$V(1/6)=3-\sqrt3$. Nonnegativity of the coefficients gives
$|U(z)|\le U(1/6)=3$ and $|V(z)|\le V(1/6)=3-\sqrt3$ throughout the closed disc
$|z|\le1/6$, whence
\[
  \bigl|z\,(1+U(z))(1+V(z))\bigr| \;\le\; \tfrac16\cdot 4\cdot\bigl(4-\sqrt3\bigr)
  \;=\; \tfrac{2(4-\sqrt3)}{3} \;=\; 1.5119\ldots\;<\;2 ,
\]
so $|D(z)|\ge 2-\tfrac{2(4-\sqrt3)}{3}=\tfrac{2(\sqrt3-1)}{3}>0$ on all of
$|z|\le1/6$. (The bound is attained at $z=1/6$, where $D(1/6)=2(\sqrt3-1)/3$.) In
particular $D$ has no zero in the closed disc, real or complex, so the only
singularity of $C$ in $|z|\le1/6$ is the branch point of $U$ at $z=1/6$: it is the
unique dominant singularity, and $C$ is amenable to singularity analysis there.

At $z=1/6$ we have $U(1/6)=3$, $V(1/6)=3-\sqrt3$, and
\[
  U(z)=3-2\sqrt3\,\sqrt{1-6z}+O(1-6z),\qquad
  \frac{\partial C}{\partial U}\Big|_{z=1/6}=\frac{2z}{D^2}\Big|_{z=1/6}=\frac{3(2+\sqrt3)}{8},
\]
whence $C(z)=\frac{3+\sqrt3}{2}-\frac{9+6\sqrt3}{4}\sqrt{1-6z}+O(1-6z)$. Standard
singularity analysis \cite[Thm.~VI.4]{FS} gives
$c_n\sim\frac{9+6\sqrt3}{4}\cdot\frac{6^nn^{-3/2}}{2\sqrt\pi}$.
\end{proof}

\begin{remark}\label{rem:bij}
The growth constant $6$ enters purely analytically, as the branch point of
$U=1+3z+zU^2$. The series $U$ is a $(1,3)$-weighted variant of the Catalan
generating function $V=1+zV^2$, but we do not know a combinatorial
interpretation, within $\W_n$ itself, of the objects counted by $U$ that would
explain the weight $3$ -- and hence no bijective explanation of the growth rate
$6$. Finding one, along with a combinatorial explanation of the appearance of
the ordinary Catalan series $V$ in Theorem~\ref{thm:main}, is an attractive
open problem.
\end{remark}

\section{Verification}\label{sec:verif}

The following independent checks were carried out.

\begin{enumerate}
\item Direct enumeration of all $(2n)!/2^n$ arrangements, filtered by the literal
  definitions (N) and (A), reproduces $c_n$ for $n\le5$.
\item A separate program building words left to right, opening arcs and closing
  the earliest open arc, with pruning on (A), gives
  \[
    \begin{gathered}
    1,\ 4,\ 17,\ 82,\ 406,\ 2070,\ 10729,\ 56394,\ 299646,\\
    1606816,\ 8683562,\ 47245644,\ 258581288,\ 1422695978
    \end{gathered}
  \]
  for $1\le n\le14$. The first eleven terms agree with A382361; the last three are
  new.
\item Theorem~\ref{thm:succ} was verified exhaustively for every object of size
  $n\le12$: for each of the $47\,245\,644$ objects of $\W_{12}$ and each smaller
  object, the full multiset of child labels was computed by inserting the new
  maximum in every one of the $\binom{2n+2}{2}$ ways and testing conditions (N)
  and (A) directly from their definitions (not via Lemma~\ref{lem:insert}), and
  compared with the rule; all agree, the label of
  Lemma~\ref{lem:shape} is always defined, and for $n\le7$ the children of $\W_n$
  were confirmed to partition $\W_{n+1}$.
\item Equations \eqref{eq:1}, \eqref{eq:4}, and \eqref{eq:5}, together with the
  identities \eqref{eq:id1}--\eqref{eq:id3}, were verified as truncated power
  series to order $z^{22}$ against the array $f_{n,r,s}$ generated by
  Theorem~\ref{thm:succ}, with $u$ and $v$ specialised to several rational values,
  in exact arithmetic.
\item The series of Theorem~\ref{thm:main} agrees with the transfer matrix of
  Theorem~\ref{thm:succ} to order $z^{45}$ in exact arithmetic, and with item (2) above for
  $n\le14$.
\item The resultant identity of Corollary~\ref{cor:alg} was verified by exact
  integer polynomial arithmetic, and the asymptotic constant of
  Corollary~\ref{cor:asy} was confirmed against $260$ exact terms by Richardson
  extrapolation.
\end{enumerate}

\section*{Acknowledgements}

The problem was posed by Elizalde and Luo \cite{EL}.

\emph{Funding.} This research received no grant from any funding agency,
commercial entity, or academic institution. All costs of the work -- hardware,
computation, and AI-assistant subscriptions -- were financed personally by the
author.

\emph{Disclosure of AI assistance.} This work was carried out by the author with
substantial assistance from large language models. The author chose the problem,
set the direction of attack, decided which lines of argument to pursue and which
to abandon, and assembled the pieces developed in separate sessions into a single
argument, contributing to the writing of those pieces along the way. The
generating-tree approach and the kernel-method derivation were developed through
iterative collaboration between the author and GPT (OpenAI). Claude (Anthropic)
assisted in expanding the derivation into a complete manuscript and in developing
and running the computational verification described in Section~\ref{sec:verif}.
The proofs, calculations, and code were subsequently subjected to separate
critical audits using GPT and Claude and were revised in response to those
audits. The verification code is supplied as supplementary material so that the
stated checks can be reproduced. An expository review by
Gemini (Google) prompted the worked example in Section~\ref{sec:eqs} and
Remark~\ref{rem:bij}. The author evaluated and revised all outputs, made the decisions
about what to pursue and publish, and takes full responsibility for the content
and its correctness.

\end{document}